\documentclass[11pt]{article}
\usepackage{fullpage,url,amssymb}
\usepackage{amsmath,  amsthm, graphicx}
 \usepackage[backref=page]{hyperref}

 \newtheorem{remark}{Remark}

\newtheorem{theorem}{Theorem}

\def\bbR{\mathbb{R}}

 \def\Sym{\mathrm{Sym}}
\def\Aff{\mathrm{Aff}}

\def\dy{\mathrm{d}y}
\def\dx{\mathrm{d}x}
\def\dmu{\mathrm{d}\mu}

\def\KL{\mathrm{KL}}

\def\bbR{\mathbb{R}}

\def\matrixtwotwo#1#2#3#4{\left[\begin{array}{ll}#1 & #2 \cr #3 & #4\end{array}\right]}

\def\tr{\mathrm{tr}}

\def\Cov{\mathhrm{Cov}}

\def\JS{\mathrm{JS}}

\def\calP{\mathcal{P}}

\def\abs#1{{\lvert #1 \rvert}}

\newtheorem{proposition}{Proposition}
\newtheorem{assumption}{Assumption}

\def\Cov{\mathrm{Cov}}

\def\arccosh{\mathrm{arccosh}}

\def\GL{\mathrm{GL}}

\def\dmu{\mathrm{d}\mu}

\def\bbR{\mathbb{R}}

\def\KL{\mathrm{KL}}

\def\dx{\mathrm{d}x}

\def\calP{\mathcal{P}}
\def\dmu{\mathrm{d}\mu}
\def\KL{\mathrm{KL}}
\def\JS{\mathrm{JS}}

\def\bbH{\mathbb{H}}

\def\tr{\mathrm{tr}}

\def\Cov{\mathrm{Cov}}

\def\det{\mathrm{det}}

\def\calP{\mathcal{P}}

\def\dt{\mathrm{d}t}

\def\TV{\mathrm{TV}}

\def\TV{\mathrm{TV}}

\def\Mdet#1{\mathrm{det}\left(#1\right)}

\author{
Frank Nielsen\\
Sony Computer Science Laboratories Inc\\
Tokyo, Japan\\
E-mail: {\tt Frank.Nielsen@acm.org}\\
\and 
Kazuki Okamura\\
Department of Mathematics, Faculty of Science, Shizuoka University\\
Japan\\
E-mail:  {\tt okamura.kazuki@shizuoka.ac.jp}
}

 \date{}

\begin{document}

\title{On the $f$-divergences between densities of a multivariate location or scale family}

\maketitle
 
\begin{abstract}
We first extend the result of Ali and Silvey [Journal of the Royal Statistical Society: Series B, 28.1 (1966), 131-142] 
who first reported that any $f$-divergence between two isotropic multivariate Gaussian distributions 
amounts to a corresponding strictly increasing scalar function of their corresponding Mahalanobis distance.
We report sufficient conditions on the standard probability density function generating a multivariate location family and the function generator $f$  in order to generalize this result.
This property is useful in practice as it allows to compare exactly $f$-divergences between densities of these location families via their  corresponding Mahalanobis distances, even when the $f$-divergences are not available in closed-form as it is the case, for example, for the
 Jensen-Shannon divergence or the total variation distance between densities of a normal location family.  
Second, we consider $f$-divergences between densities of multivariate scale families: We recall Ali and Silvey 's result that for normal scale families we get matrix spectral divergences, and we extend this result to densities of a scale family.
\end{abstract}


\section{Introduction} 
Let $\bbR$ denote the real field and $\bbR_{++}$ the set of positive reals.
The $f$-divergence~\cite{csiszar1964informationstheoretische,ali1966general} induced by a convex generator $f:\bbR_{++}\rightarrow \bbR$ between two probability density functions (PDFs) $p(x)$ and $q(x)$ defined on the full support $\bbR^d$ is defined by
$$
I_f(p:q) := \int p(x)\, f\left(\frac{q(x)}{p(x)}\right)\, \dx. 
$$
It follows from Jensen's inequality that we have 
$$
I_f(p:q)\geq f\left(\int p(x)\frac{q(x)}{p(x)} \dx \right) =f(1).
$$
Thus we shall consider convex generators $f(u)$ such that $f(1)=0$. 
Moreover, in order to ensure that $I_f(p:q)=0$ if and only if ${p(x)=q(x)}$ except at countably many points $x$, we need $f(u)$ to be strictly convex at $u=1$. 
The class of $f$-divergences include the total variation distance (the only $f$-divergence up to scaling which is a metric distance~\cite{khosravifard2007confliction}), the Kullback-Leibler divergence (and its two common symmetrizations, namely, the Jeffreys divergence and the Jensen-Shannon divergence), the squared Hellinger divergence, the Pearson and Neyman $\chi^2$-divergences, etc.
The formula for those statistical divergences  with their corresponding generators are listed in Table~\ref{tab:fgenerator}.

\begin{table}
\renewcommand{\arraystretch}{1.5}
\centering
 $$
\begin{array}{lll}
\text{$f$-divergence } & \text{Formula $I_f(p:q)$} & \text {Generator $f(u)$ with $f(1)=0$}\\
\hline\hline
\text{Total variation} & \frac{1}{2}\int \lvert p(x)-q(x)\rvert \dx & \frac{1}{2} \lvert u-1\rvert \\
\text{Squared Hellinger} & \int (\sqrt{p(x)}-\sqrt{q(x)})^2 \dx & (\sqrt{u}-1)^2\\
\text{Pearson $\chi^2$}  &  \int \frac{(q(x)-p(x))^2}{p(x)} \dx & (u-1)^2\\
\text{Neyman $\chi^2$}  &  \int \frac{(p(x)-q(x))^2}{q(x)} \dx & \frac{(1-u)^2}{u}\\
\text{Kullback-Leibler} & \int p(x)\log \frac{p(x)}{q(x)} \dx & -\log u\\
\text{reverse Kullback-Leibler} & \int q(x)\log \frac{q(x)}{p(x)} \dx & u\log u \\
\text{Jeffreys} & \int (p(x)-q(x))\log \frac{p(x)}{q(x)} \dx & (u-1)\log u \\
\text{Jensen-Shannon} & h\left(\frac{p+q}{2}\right)-\frac{h(p)+h(q)}{2} &  -(u+1)\log \frac{1+u}{2} + u\log u \\
 & \multicolumn{2}{c}{\text{where\ } h(p)=\int p(x)\log\frac{1}{p(x)}\dx \text{\ is\ Shannon entropy}}
\end{array}
$$
\caption{Some common statistical divergences expressed as $f$-divergences.\label{tab:fgenerator}}
\end{table}

Let $N(\mu,\Sigma)\sim p_{\mu,\Sigma}(x)$ be a multivariate normal  (MVN) distribution with mean $\mu\in\bbR^d$ and positive-definite covariance matrix $\Sigma\in\Sym_{++}(d)$ (where $\Sym_{++}(d)$ denotes the set of positive-definite matrices), where the PDF is defined by
$$
p_{\mu,\Sigma}(x)=\frac{1}{\sqrt{\det(2\pi\Sigma)}}\, \exp\left(-\frac{1}{2}(x-\mu)^\top \Sigma^{-1} (x-\mu)\right).
$$ 
In their landmark paper, Ali and Silvey~\cite{ali1966general} (Section~6 of ~\cite{ali1966general}, pp. 141-142) mentioned 
 the following two  properties of $f$-divergences between MVN distributions:

\begin{enumerate}
	\item[P1.] The $f$-divergences $I_f(p_{\mu_1,\Sigma}:p_{\mu_2,\Sigma})$ between two MVN distributions $N(\mu_1,\Sigma)$ and $N(\mu_2,\Sigma)$ with prescribed covariance matrix $\Sigma$  is an increasing function of their Mahalanobis distance~\cite{Mahalanobis-1936} $\Delta_\Sigma(\mu_1,\mu_2)$, where
	$$
	\Delta_\Sigma^2(\mu_1,\mu_2):=(\mu_2-\mu_1)^\top\, \Sigma^{-1} (\mu_2-\mu_1).
	$$
Ali and Silvey briefly sketched a proof by considering the following property obtained from a change of variable 
$y=\frac{(x-\mu_1)^\top \Sigma^{-1} (\mu_2-\mu_1)}{\Delta_\Sigma(\mu_1,\mu_2)}\in\bbR$:
	$$
	I_f(p_{\mu_1,\Sigma}:p_{\mu_2,\Sigma})=I_f(p_{0,1}:p_{\Delta_\Sigma(\mu_1,\mu_2),1}).
	$$
	That is, the $f$-divergences between multivariate normal distributions with prescribed covariance matrix (left-hand side) amount to corresponding
	$f$-divergences between the univariate normal distribution $N(0,1)$ and $N(\Delta_\Sigma(\mu_1,\mu_2),1)$ (right-hand side).
	
	\item[P2.] The $f$-divergences $I_f(p_{\mu,\Sigma_1}:p_{\mu,\Sigma_1})$ between two $\mu$-centered MVN distributions $N(\mu,\Sigma_1)$ and 
	$N(\mu,\Sigma_2)$ is an increasing function of the terms 
	$\abs{1-\lambda_i}$'s, 
	where the $\lambda_i$'s denote the eigenvalues of matrix $\Sigma_2\Sigma_1^{-1}$.
	That is, the $f$-divergences between $\mu$-centered MVN distributions are spectral matrix divergences~\cite{kulis2009low}.
\end{enumerate}

In this paper, we investigate whether these two properties hold or not for multivariate location families and multivariate scale families which generalize the multivariate centered (same mean) normal families $N_\mu:=\{N(\mu,\Sigma)\ :\ \Sigma\in\Sym_{++}(d)\}$ 
and the multivariate isotropic (same covariance) normal distributions $N_\Sigma:=\{N(\mu,\Sigma)\ :\  \mu\in\bbR^d\}$, respectively.

We summarize our main contributions as follows:
\begin{itemize}
	\item We extend property P1 to arbitrary multivariate location families in Theorem~\ref{prop:mc} under Assumption~\ref{ass:regularity}      (i.e., spherical distribution of the standard PDF with $f(u)\in C^2$).
	\item We illustrate property P1 for the multivariate location normal distributions and the multivariate location Cauchy distributions for various $f$-divergences, and discuss practical computational applications in~Section~\ref{sec:mvuv}.
	\item We then report the spectral matrix $f$-divergences (Property P2) for the multivariate scale normal distributions for the Kullback-Leibler divergence and the $\alpha$-divergences in Section~\ref{sec:ns}, and generalize this property to 
	densities of a scale family (Proposition~\ref{prop:spectralfdiv}).	
\end{itemize}

The paper is organized as follows:
We first describe generic multivariate location-scale families including the action of the affine group on $f$-divergences in Section~\ref{sec:fdivmls1}. 
In Section~\ref{sec:fdivmls2}, we present our main theorem which generalizes property P1,  and illustrate the theorem with examples of $f$-divergences between multivariate location normal or location Cauchy families.
Finally, we show that the $f$-divergences between densities of a scale family is always a spectral matrix divergence (
Proposition~\ref{prop:spectralfdiv}).

\section{The $f$-divergences between multivariate location-scale families}\label{sec:fdivmls1}

\subsection{Multivariate location-scale families}
Let $p(x)$ be any arbitrary PDF defined on the full support $\bbR^d$.
The multivariate location-scale family $\calP$ is defined as the set of distributions with PDFs:
$$
p_{l,P}(x):= \det(P)^{-1}\, p\left( P^{-1}(x-l) \right), x \in \mathbb{R}^d.
$$ 
where the set of multivariate location-scale parameters $(l,P)$ belongs to $\bbH_d:=\bbR^d\times \Sym_{++}(d)$.
PDF $p(x)$ is called the standard density since $p(x)=p_{0,I}(x)$  where $I$ denote the $d$-dimensional identity matrix.

By considering $P=(\Sigma^{\frac{1}{2}})^2$ where $\Sigma^{\frac{1}{2}}$ denotes the unique square root of the covariance matrix $\Sigma$, and letting $l=\mu$, we can express
the PDFs of $\calP$ as 
\begin{equation}\label{eq:fdiv}
p_{\mu, \Sigma}(x) = \frac{1}{\sqrt{\det(\Sigma)}}  \,  p\left(\Sigma^{-1/2}(x-\mu)\right), x \in \mathbb{R}^d.
\end{equation}

The multivariate location-scale families generalize the univariate location-scale families with $\Sigma=\sigma^2$: 
$$
\calP=\left\{ \frac{1}{\sigma} p\left(\frac{x-\mu}{\sigma}\right) \ :\ (\mu,\sigma)\in\bbR\times\bbR_{++}\right\}.
$$

In the reminder, we shall focus on the following two multivariate location-scale families:
\begin{enumerate}
\item[MVN.] MultiVariate Normal location families: 
$$
p(x) = \frac{1}{(2\pi)^{d/2}} \exp\left(-\frac{x}{2}\right), \ x \ge 0.
$$

Notice that for $X\sim N(0,I_d)$, we have $Y=PX+l\sim N(l,PP^\top)$. 
Thus when $l=\mu$ and $P=\Sigma^{\frac{1}{2}}$, we get $Y\sim N(\mu,\Sigma)$.

\item[MVS.] MultiVariate Student location families with $\nu$ degree(s) of freedom:
$$
p(x) = \frac{\Gamma((\nu+d)/2)}{\Gamma(\nu/2) (\nu\pi)^{d/2}} \frac{1}{(1+x/\nu)^{(\nu + d)/2}}, \ x \ge 0, 
$$
where $\Gamma(t)$ is the Gamma function:
$$
\Gamma(t)=\int_0^\infty x^{t-1} e^{-x} \dx.
$$
The case of $\nu=1$ corresponds to the MultiVariate Cauchy (MVC) family.
The probability density function of a MVC is
$$
p_{\mu,\Sigma}(x)=\frac{\Gamma\left(\frac{d+1}{2}\right)}{\pi^{\frac{d+1}{2}}\Mdet{\Sigma}^{\frac{1}{2}} (1+\Delta_\Sigma(x,\mu))^{\frac{d+1}{2}}}.
$$

Let us note in passing that if random vector $X = (X_1, \dots, X_d)$ follows the standard MVC, then $X_1, \dots, X_d$ are not statistically independent~\cite{molenberghs1997non}. Thus the MVC family differs from the MVN family from that viewpoint.
\end{enumerate}

Multivariate location-scale families also include the multivariate elliptical distributions~\cite{nielsen2021information,kollo2005advanced}.

\subsection{Action of the affine group}
The family $\calP=\{p_{l,P}\ :\ (l,P)\in\bbH\}$ can also be obtained by the action (denoted by the dot $.$) of the affine group $\Aff(\bbR^d):=\bbR^d \rtimes \GL_d(\bbR)$~\cite{globke2021information} (where $\GL_d(\bbR)$ denotes the General Linear $d$-dimensional group) on the standard PDF: $\calP=\{p_{l,P}(x)=(l,P).p(x), (l,P)\in \Aff(\bbR^d)\}$,
where the group is equipped with the following semidirect product:
$$
(l_1,A_1).(l_2,A_2):=(l_1+A_1l_2, A_1A_2).
$$  
The inverse element of $(l,A)$ is $(-A^{-1}l,A^{-1})$.
One can check that
$(l,A).(-A^{-1}l,A^{-1})=(-A^{-1}l,A^{-1}).(l,A)=(0,I_d)$, where $I$ denotes the $d\times d$ identity matrix.

The affine group $\Aff(\bbR^d)$ can be handled as a matrix group by mapping its elements $(l,A)$ to corresponding $(d+1)\times (d+1)$ matrices as follows:
$$
(l,A)\leftrightarrow \matrixtwotwo{A}{l}{0}{1}.
$$
That is, $\Aff(\bbR^d)$ can be interpreted as a subgroup of $\GL(d+1)$.

\subsection{Affine group action on $f$-divergences}

The $f$-divergences 
between two PDFs $p_{l_1,P_1}$ and $p_{l_2,P_2}$ 
of a multivariate location-scale family $\calP$ are invariant under the action of the affine group~\cite{nielsen2021information}:
$$
I_f\left((l,P).p_{l_1,P_1}:(l,P).p_{l_2,P_2}\right) = I_f\left(p_{l_1,P_1}:p_{l_2,P_2}\right).
$$

Thus by choosing the inverse element $(-P_1^{-1}l_1,P_1^{-1})$ of $(l_1,P_1)$, we get the following Proposition~\cite{nielsen2021information}:

\begin{proposition}\label{prop:fdivls}
We have
\begin{eqnarray*}
I_f\left(p_{l_1,P_1}:p_{l_2,P_2}\right) &=&
I_f\left((-P^{-1}_1l,P_1^{-1}).p_{l_1,P_1}:(-P_1^{-1}l_1,P_1^{-1}).p_{l_2,P_2}\right),\\
&=& I_f\left(p_{0,I}:p_{P_1^{-1}(l_2-l_1),P_1^{-1}P_2}\right).
\end{eqnarray*}

\begin{remark}
When $f(u)=-\log u$, we get the Kullback-Leiber divergence and we have:
$$
D_\KL\left(p_{l_1,P_1}:p_{l_2,P_2}\right)=D_\KL\left(p_{0,I}:p_{P_1^{-1}(l_2-l_1),P_1^{-1}P_2}\right).
$$
Since the KLD is the difference of the cross-entropy minus the entropy~\cite{CT-1999} (hence also called relative entropy), we have
$$
D_\KL\left(p_{l_1,P_1}:p_{l_2,P_2}\right)=h\left(p_{0,I}:p_{P_1^{-1}(l_2-l_1),P_1^{-1}P_2}\right)-h(p_{0,I}),
$$
where $h(p:q)=-\int p(x)\log q(x)\dx$ is the cross-entropy between $p(x)$ and $q(x)$, and $h(p)=h(p:q)$ is the differentiable entropy.
When both $p(x)=p_{\mu_1,\Sigma_1}(x)$ and $q(x)=p_{\mu_2,\Sigma_2}(x)$ are $d$-variate normal distributions, the KLD can be decomposed as the sum a squared Mahalanobis distance $\Delta_{\Sigma_2^{-1}}(\mu_1,\mu_2)$ plus a matrix Burg divergence~\cite{davis2006differential} $D_B(\Sigma_1,\Sigma_2)$:
$$
D_\KL(p_{\mu_1,\Sigma_1}:p_{\mu_2,\Sigma_2})=\frac{1}{2}D_B(\Sigma_1,\Sigma_2)+\frac{1}{2}\Delta_{\Sigma_2^{-1}}(\mu_1,\mu_2),
$$
where the matrix Burg divergence is defined by
$$
D_B(\Sigma_1,\Sigma_2)=\tr(\Sigma_2\Sigma_1^{-1})+\log\Mdet{\Sigma_2\Sigma_1^{-1}}-d.
$$
However, the KLD between two Cauchy distributions  (viewed like normal distributions as  a location-scale family) cannot be decomposed as the sum of a squared Mahalanobis distance and another divergence~\cite{nielsen2021f}.
\end{remark}

In particular, we have for PDFs with the same scale matrix $P$, the following identity:
$$
I_f(p_{l_1,P}:p_{l_2,P})=I_f \left(p_{0,I}:p_{P(l_2-l_1)}\right)
 = I_f\left(p_{P(l_1-l_2)}:p_{0,I}\right).
$$
\end{proposition}

Let $\Sigma:=PP^\top$ so that $P=\Sigma^{\frac{1}{2}}$.
The mapping $P\leftrightarrow\Sigma^{\frac{1}{2}}$ is a diffeomorphism on the open cone $\Sym_{++}(d)$ of positive-definite matrices.

Thus we have
$$
I_f(p_{\mu_1,\Sigma}:p_{\mu_2,\Sigma})=I_f \left(p_{0,I}:p_{\Sigma^{-\frac{1}{2}}(\mu_2-\mu_1)}\right)
 = I_f\left(p_{\Sigma^{-\frac{1}{2}}(\mu_1-\mu_2)}:p_{0,I}\right).
$$

Note that in general, we have $\mu\not=E_{X\sim p_{\mu,\Sigma}}[X]$ and $\Sigma\not=\Cov_{X\sim p_{l,P}}[X]$.
However, for the special case of multivariate normal family, we have both  $\mu=E_{X\sim p_{\mu,\Sigma}}[X]$ and $\Sigma=\Cov_{X\sim p_{l,P}}[X]$.

\section{The $f$-divergences between densities of a multivariate location family}\label{sec:fdivmls2}
Let us define the squared Mahalanobis distance~\cite{Mahalanobis-1936} between two MVNs
 $N(\mu_1,\Sigma)$ and $N(\mu_2,\Sigma)$ as follows:
$$
\Delta_\Sigma^2(\mu_1,\mu_2) := (\mu_2-\mu_1)^\top\, \Sigma^{-1} (\mu_2-\mu_1).
$$
Since the covariance matrix $\Sigma$ is positive-definite, we have $\Delta_\Sigma^2(\mu_1,\mu_2)\geq 0$ and zero if and only if $\mu_1=\mu_2$.
The squared Mahalanobis distance generalizes the squared Euclidean distance when $\Sigma=I$: 
$\Delta_I^2(\mu_1,\mu_2)=\|\mu_1-\mu_2\|^2$.
 
The Kullback-Leibler divergence~\cite{CT-1999} between two multivariate isotropic normal distributions corresponds to half 
the squared Mahalanobis distance:
$$
D_\KL(p_{\mu_1,\Sigma}:p_{\mu_2,\Sigma})=\int_{x\in\bbR^d} p_{\mu_1,\Sigma}(x)\log\frac{p_{\mu_1,\Sigma}(x)}{p_{\mu_2,\Sigma}(x)} \dx= \frac{1}{2} \Delta_\Sigma^2(\mu_1,\mu_2),
$$
where 
$$
D_\KL(p:q)=\int p(x)\log\left(\frac{p(x)}{q(x)}\right) \dmu(x)=I_{f_\KL}(p:q)
$$
for $f_\KL(u)=-\log u$.

Moreover, the PDD of a MVN with covariance matrix $\Sigma$ and mean $\mu$ can also be written using the squared Mahalanobis distance as follows:
\begin{eqnarray*}
p(x;\mu,\Sigma)= \frac{1}{\sqrt{\det(2\pi \Sigma)}}\, \exp\left(-\frac{1}{2}\Delta_\Sigma^2(x,\mu)\right).
\end{eqnarray*}
This  rewriting highlights the general duality between Bregman divergences (e.g., the squared Mahalanobis distance is a Bregman divergence) and the exponential families~\cite{banerjee2005clustering} (e.g., multivariate Gaussian distributions).

We shall make the following set of assumptions for the standard density  $p(x)$ and $f(u)$:

\begin{assumption}\label{ass:regularity}
(i) We assume that there exists a function $p : [0, \infty) \to (0, \infty)$ such that $p$ is in $C^1$ class, $p^{\prime}(x) < 0, x \geq 0$, 
and furthermore $p(x) = \widetilde p(\|x\|^2), \ \ x \in \mathbb{R}^d$.\\
(ii) We assume that $f : (0,\infty) \to \mathbb R$ satisfies that it is in $C^2$ class, $f(1) = 0$ and $f^{\prime\prime} (x) > 0, x > 0$.\\
(iii) For every $t \in \mathbb{R}^d$, 
$$
\int_{\mathbb R^d} \left \lvert f\left(\frac{\widetilde p(\|x+t\|^2)}{\widetilde p(\|x\|^2)}\right)\right\rvert \, \widetilde p(\|x\|^2)\, dx < +\infty.  
$$
(iv) For every compact subset $K$ of $\mathbb{R}^d$,
$$
 \int_{\mathbb R^d} \sup_{t \in K} \left\rvert f^{\prime} \left( \frac{\widetilde p(\|y\|^2)}{\widetilde p(\|y + t \|^2)} \right) \right\lvert \left\lvert \widetilde p^{\, \prime}(\|y\|^2) \right\rvert \|y\| \dy < +\infty.
$$
\end{assumption}

Notice that using the assumption that the standard density $p(x)=\tilde{p}(\|x^2\|)$, we get
\begin{eqnarray*}
p_{\mu, \Sigma}(x) &=& \frac{1}{\sqrt{\det(\Sigma)}}  \,  p\left(\Sigma^{-1/2}(x-\mu)\right),  \\
&=& \frac{1}{\sqrt{\det(\Sigma)}}  \, \tilde{p}(\Delta_\Sigma^2(x,\mu)),
\end{eqnarray*}
since $\lvert \Sigma^{-\frac{1}{2}} (x-\mu)\rvert^2 = (\Sigma^{-\frac{1}{2}} (x-\mu))^\top  \Sigma^{-\frac{1}{2}} (x-\mu) =\Delta_\Sigma^2(x,\mu)$.
Thus the standard density $p(x)$ of the multivariate location-scale family $\mathcal{P}$ is the density of a spherical distribution~\cite{steerneman2005spherical} and $p_{\mu, \Sigma}(x)$ is the density of  an elliptical symmetric distribution~\cite{ollila2012complex}.

We let 
$$
I_f \left(p_{\mu_1, \Sigma} : p_{\mu_2, \Sigma} \right) := \int_{\mathbb R^d} f\left(\frac{p_{\mu_2, \Sigma}(x)}{p_{\mu_1, \Sigma}(x)}\right) p_{\mu_1, \Sigma}(x) \dx.
$$
This is well-defined due to Assumption \ref{ass:regularity} (iii). 

We state the main theorem generalizing~\cite{ali1966general}:

\begin{theorem}[$f$-divergence between location families]\label{prop:mc}
Under Assumption \ref{ass:regularity}, there exists a strictly increasing and differentiable function $h_f$ such that 
\begin{equation}\label{eq:mc}
 I_f \left(p_{\mu_1, \Sigma} : p_{\mu_2, \Sigma} \right)= h_f \left( \Delta_\Sigma^2(\mu_1,\mu_2) \right), \ \ \mu_1, \mu_2 \in \mathbb{R}^d. 
\end{equation}
\end{theorem}

\begin{proof}

\begin{description}
\item [\underline{Step 1}.] We let $d=1$ and $\Sigma = 1$. 
\[ I_f \left(p_{\mu_1, 1} : p_{\mu_2, 1} \right) = I_f \left(p_{0, \Sigma} : p_{\mu_2 - \mu_1, 1} \right)  = \int_{\mathbb R} f\left(\frac{\widetilde p((x-(\mu_2 - \mu_1))^2)}{\widetilde p(x^2)}\right) \widetilde p(x^2) dx. \]
Let 
\[ F(t) := I_f \left(p_{0, 1} : p_{t, 1} \right) = \int_{\mathbb R} f\left(\frac{\widetilde p((x-t)^2)}{\widetilde p(x^2)}\right) \widetilde p(x^2) dx, \ \ t \in \mathbb{R}.  \]
By Assumption \ref{ass:regularity} (i), $F(t) = F(-t), \ \ t \in \mathbb{R}$. 
Hence if we let $h_f (s) := F(\sqrt{s})$, then $h_f$ satisfies Eq.~(\ref{eq:mc}). 

Now it suffices to show that $F$ is strictly increasing and differentiable. 
By Assumption \ref{ass:regularity} (iv), 
\begin{align*} 
F^{\prime}(t) &=  \int_{\mathbb R} \frac{d}{\mathrm{d}t} \left(f\left(\frac{\widetilde p((x-t)^2)}{\widetilde p(x^2)}\right) \right) \widetilde p(x^2) dx \\
&= 2  \int_{\mathbb R} f^{\prime}\left(\frac{\widetilde p((x-t)^2)}{\widetilde p(x^2)}\right) (t-x) \widetilde p^{\, \prime}((x-t)^2) dx\\
&= -2  \int_{\mathbb R} f^{\prime}\left(\frac{\widetilde p(y^2)}{\widetilde p((y+t)^2)}\right) y \widetilde p^{\, \prime}(y^2) dy\\
&= -2  \int_{0}^{\infty} \left(f^{\prime}\left(\frac{\widetilde p(y^2)}{\widetilde p((y+t)^2)}\right) - f^{\prime}\left(\frac{\widetilde p(y^2)}{\widetilde p((y-t)^2)}\right)  \right) y \widetilde p^{\, \prime}(y^2) dy. 
\end{align*}

By the mean-value theorem and Assumption \ref{ass:regularity} (i) and (ii), 
$$
f^{\prime}\left(\frac{\widetilde p(y^2)}{\widetilde p((y+t)^2)}\right) - f^{\prime}\left(\frac{\widetilde p(y^2)}{\widetilde p((y-t)^2)} \right) > 0,\quad \widetilde p^{\, \prime}(y^2) < 0, 
$$
for every $y, t > 0$. 
Hence $F^{\prime}(t) > 0$ for every $t > 0$. 

\item[\underline{Step 2}.]
 We let $\Sigma = I$. 
\[ I_f \left(p_{\mu_1, I_d} : p_{\mu_2, I_d} \right) = I_f \left(p_{0, \Sigma} : p_{\mu_2 - \mu_1, I_d} \right)  = \int_{\mathbb R^d} f\left(\frac{\widetilde p(\|x-(\mu_2 - \mu_1)\|^2)}{\widetilde p(\|x\|^2)}\right) \widetilde p(\|x\|^2) dx. \]
Let 
\begin{equation}\label{eq:def-F} 
F(t) := I_f \left(p_{0, 1} : p_{t, I_d} \right) = \int_{\mathbb R^d} f\left(\frac{\widetilde p(\|x-t\|^2)}{\widetilde p(\|x\|^2)}\right) \widetilde p(\|x\|^2) dx, \ \ t \in \mathbb{R}^d.  
\end{equation}
By changing the variable $x$ by an orthogonal matrix, 
$F(s) = F(t)$ if $\|s\| = \|t\|$.  
Hence we can assume that $t = (t,0, \dots, 0)^\top \in \mathbb{R}^d, t > 0$. 
For $x_2, \dots, x_d \in \mathbb R$, let 
\[ F_{x_2, \dots, x_d}(t) := \int_{\mathbb R} f\left(\frac{\widetilde p((x_1 - t)^2 + x_2^2 + \cdots + x_d^2)}{\widetilde p(\|x\|^2)}\right) \widetilde p(\|x\|^2) dx_1, \ \ t \in \mathbb{R}.\]
Then, we can show that $F_{x_2, \dots, x_d}^{\prime}(t) > 0$ for every $t > 0$ and every $x_2, \dots, x_d \in \mathbb R$, in the same manner as in Step 1. 
By this and Assumption \ref{ass:regularity} (iv), 
\[ F^{\prime}(t) = \int_{\mathbb{R}^{d-1}} F_{x_2, \dots, x_d}^{\prime}(t) dx_2 \dots dx_d > 0, \ \ t > 0. \]
Hence if we let $h_f (s) := F(\sqrt{s})$, then $h_f$ satisfies Eq.~(\ref{eq:mc}). 

\item [\underline{Step 3}.] Finally, we consider the general case.  
Let $\mu := \Sigma^{-1/2} (\mu_2 - \mu_1)$.  
Then, 
\begin{equation*}
I_f \left(p_{\mu_1, \Sigma} : p_{\mu_2, \Sigma} \right) = \sqrt{\det \Sigma}\, I_f \left(p_{0, I_d} : p_{\mu, I_d} \right). 
\end{equation*}
Hence this case is attributed to Step 2. 
\end{description}

\end{proof}

We shall now illustrate this theorem with several examples.

\subsection{The normal location families}

\subsubsection{The scalar functions $h_f$ for some common $f$-divergences}

Table~\ref{tab:MahIsoMVN} lists some examples of $f$-divergences with their corresponding monotone increasing functions $h_f$.
We consider the following $f$-divergences between two PDFs of MVNs with same covariance matrix:

\begin{itemize}
\item For the $\chi^2$ divergence with $f_{\chi,2}(u)=(u-1)^2$, we have $h_{\chi^2}(u)=h_{\chi,2}(u)=1-\exp\left(-\frac{1}{2}u\right)$, 
and more generally for the order-$k$ chi divergences ($f$-divergence  generator $f_{\chi,k}(u)=(u-1)^k$) between $p_{\mu_1,\Sigma}$ and $p_{\mu_2,\Sigma}$, we get~\cite{nielsen2013chi}:
$$
h_{\chi,k}(u)= \sum_{i=0}^k (-1)^{k-i} \binom{k}{i} \exp\left(\frac{1}{2}i(i-1) \, \Delta_\Sigma^2(\mu_1,\mu_2)\right).
$$
We observe that the order-$k$ $\chi$-divergence between isotropic MVNs diverges as $k$ increases.
It is easy to check that we can compute $h_f$ in closed form for any convex polynomial $f$-divergence generator $f(u)$.  

\item For the Kullback-Leibler divergence, we have $D_\KL[p_{\mu_1,\Sigma}:p_{\mu_2,\Sigma}]=\frac{1}{2}\Delta_\Sigma^2(\mu_1,\mu_2)$ so $h_\KL(u)=\frac{1}{2}u$. 
Notice that because $f$-divergences between isotropic MVNs are symmetric, the Chernoff information coincides with the Bhattacharyya distance.

\item The total variation divergence (a metric $f$-divergence obtained for $f_\TV(u)=\lvert u-1\rvert$ which is always upper bounded by $1$) between two multivariate Gaussians 
$p_{\mu_1,\Sigma}$ and $p_{\mu_2,\Sigma}$ with the same covariance matrix is reported indirectly in~\cite{rohban2013impossibility}: 
The probability of error $P_e$ (with $P_e\leq\frac{1}{2}$) in Bayesian binary hypothesis with equal prior is
$P_e(p_1,p_2)=\frac{1}{2}(1-D_\TV[p_1,p_2])=Q\left(\frac{1}{2} \|\Sigma^{-\frac{1}{2}}(\mu_2-\mu_1)\|\right)$ (Eq.~(2) in~\cite{rohban2013impossibility}) where $Q(x)=1-Q(-x)=1-\Phi(x)$ where $\Phi(x)$ denotes the cumulative distribution function of the standard normal distribution.
So we get the function $h_\TV$ as a definite integral of a function of a squared Mahalanobis distance:
\begin{equation*}
D_\TV[p_1,p_2]=1-2\, P_e(p_1,p_2)=1-2\,Q\left(\frac{1}{2}\sqrt{\Delta_\Sigma^2(\mu_1,\mu_2)}\right).
\end{equation*}

\end{itemize}

\renewcommand{\arraystretch}{1.5}
\begin{table} 
\centering
\begin{tabular}{ll}
$f$-divergence & $f(u)$ and $h_f(u)$\\ \hline
$\chi$-squared divergence & $(u-1)^2$  and $1-\exp\left(-\frac{1}{2}u\right)$\\
Order-$k$ $\chi$ divergence & $(u-1)^k$ and $\sum_{i=0}^k (-1)^{k-i} \binom{k}{i} \exp\left(\frac{1}{2}i(i-1)u\right)$ \\
Kullback-Leibler divergence & $-\log(u)$  and $\frac{1}{2}u$\\
squared Hellinger divergence & $(\sqrt{u}-1)^2$ and $1-\exp\left(-\frac{1}{8}u\right)$ \\
Amari's $\alpha$-divergence & 
 $\frac{4}{1-\alpha^2} \left(1-u^{\frac{1+\alpha}{2}}\right)$ 
and $\frac{4}{1-\alpha^2}\left(1-\exp(-\frac{1-\alpha^2}{8}u)\right)$\\
Jensen-Shannon divergence & $u\log u-(1+u)\log\frac{1+u}{2}$ and $\frac{1}{4}{u}-I_\JS(u)$\\ 
Total variation distance &  $\lvert u-1\rvert$ and $1-2 Q(\frac{1}{2}\sqrt{u}):=1-2\int_{\frac{1}{2}\sqrt{u}}^{+\infty} \frac{1}{\sqrt{2\pi}}\exp(-\frac{1}{2}t^2) \dt$ \\ \hline
\end{tabular}
\caption{The $f$-divergences between two normal distributions with identical covariance matrix can always be expressed as an increasing function of the squared Mahalanobis distance: 
$I_f(p_{\mu_1,\Sigma}:p_{\mu_2,\Sigma})=h_f(\Delta_\Sigma^2(\mu_1,\mu_2))$.}
\label{tab:MahIsoMVN}
\end{table}
 
\begin{remark}
Notice that the Fisher-Rao distance between two multivariate normal distributions with the same covariance matrix is also a monotonic increasing function of their Mahalanobis distance~\cite{MVNGeodesic-1991}:
\begin{equation*}
\rho(p_{\mu_1,\Sigma},p_{\mu_2,\Sigma})=\sqrt{2}\, \arccosh \left( 1+ \frac{\Delta_\Sigma^2(\mu_1,\mu_2)}{4}\right),
\end{equation*}
where
$\arccosh(x)=\log(x+\sqrt{x^2-1})$ for $x\geq 1$.
That is, we have $h_\rho(u)=\sqrt{2}\, \arccosh \left( 1+ \frac{u}{4}\right)$.
\end{remark}

\subsubsection{The special case of the Jensen-Shannon divergence}\label{sec:JSD}
The Jensen-Shannon divergence~\cite{Lin-1991} is a symmetrization of the Kullback-Leibler divergence:
$$
D_\JS[p,q]=\frac{1}{2}\left(D_\KL\left[p:\frac{p+q}{2}\right] + D_\KL\left[q:\frac{p+q}{2}\right] \right).
$$
The JSD is a $f$-divergence for the generator $f_\JS(u)=u\log u-(1+u)\log\frac{1+u}{2}$, is always upper bounded by $\log 2$, and can further be embedded into a Hilbert space~\cite{JSD-2004}.
The JSD can be interpreted in information theory as the transmission rate in a discrete memoryless channel~\cite{JSD-2004}.
 
Although we do not have a closed-form formula for the  Jensen-Shannon divergence $D_\JS[p_{\mu_1,\Sigma}:p_{\mu_2,\Sigma}]$,
knowing that $D_\JS[p_{\mu_1,\Sigma}:p_{\mu_2,\Sigma}]=h_{f_\JS}(\Delta^2_\Sigma(\mu_1,\mu_2))$, allows one to compare exactly the JSDs since $f_\JS$ is a strictly increasing function.
That is, we have the equivalence of following signs of the predicates:
$$
D_\JS[p_{\mu_1,\Sigma}:p_{\mu_2,\Sigma}] > D_\JS[p_{\mu_3,\Sigma}:p_{\mu_4,\Sigma}] \Leftrightarrow \Delta^2_\Sigma(\mu_1,\mu_2)>\Delta^2_\Sigma(\mu_3,\mu_4).
$$

In~\cite{Entropy-2GMM-2008}, 
a formula for the differential entropy of the Gaussian mixture $m(x;\mu,\sigma)=\frac{1}{2}p_{-\mu,\sigma}(x)+\frac{1}{2}p_{-\mu,\sigma}(x)$ is reported using a definite integral which we translate using the squared Mahalanobis distance as follows:
$$
h(m(x;\mu,\sigma))=\frac{1}{2}\log(2\pi e\sigma^2)+\frac{1}{4}{\Delta_{\sigma^2}(-\mu,\mu)}-I_\JS(\Delta_{\sigma^2}(-\mu,\mu)),
$$
where
$$
I_\JS(\Delta^2):=\sqrt{\frac{8}{\pi\Delta^2}} \, \exp\left(-\frac{\Delta^2}{8}\right)\, \int_0^{\infty} e^{-\frac{2x}{\Delta^2}}\, \cosh(x)\,\log\cosh(x) \dx.
$$
We have $I_\JS(0)=0$ and the function $I_\JS$ can be tabulated as in~\cite{Entropy-2GMM-2008}.

Since the Jensen-Shannon divergence between two distributions amounts to the differential entropy of the mixture minus the average of the mixture entropies, we get 
$$
D_\JS[p_{\mu_1,\Sigma},p_{\mu_2,\Sigma}] = D_\JS[p_{0,1},p_{\Delta^2_\Sigma(\mu_1,\mu_2),1}]
$$
\begin{eqnarray*}
&=& \frac{1}{2}\log\det(2\pi e\Sigma) +\frac{1}{4}{\Delta_{\Sigma^2}(\mu_1,\mu_2)}-I_\JS(\Delta_{\Sigma^2}(\mu_1,\mu_2))- \frac{1}{2} \log \det(2\pi e\Sigma)\nonumber\\ 
&=& \frac{1}{4}{\Delta_{\Sigma^2}(\mu_1,\mu_2)} -I_\JS(\Delta_{\Sigma^2}(\mu_1,\mu_2)), 
\end{eqnarray*}
since $h[p_{\mu_1,\Sigma}]=h[p_{\mu_2,\Sigma}]=\frac{1}{2}\log \det( 2\pi e\Sigma) =\frac{d}{2}\log(2\pi e)  +\frac{1}{2}\log \vert\Sigma\rvert$.

In Section~\ref{sec:mvuv}, we graph the functions $h_f$ for the total variation distance and the Jensen-Shannon divergence.

\subsection{Cauchy location family}

Notice that for $d=1$, since the $\chi^2$-divergence (a $f$-divergence for $f(u)=(u-1)^2$) between
two Cauchy location densities $p_{l_1,s}$ and $p_{l_2,s}$ with prescribed scale $s$ is~\cite{nielsen2021f}:
$$
D_{\chi^2}(p_{l_1,s}:p_{l_2,s})=\frac{(l_2-l_1)^2}{2s^2}:=\chi_s(l_1,l_2),
$$
we have
$$
D_{\chi^2}(p_{l_1,s}:p_{l_2,s})=\frac{1}{2}\Delta^2_{s^2}(l_1,l_2),
$$
where $\Delta^2_{s^2}(l_1,l_2)=\frac{(l_2-l_1)^2}{s^2}$.
Thus it follows that $h_{f_\chi}(u)=\frac{1}{2}u^2$.

Now, since any $f$-divergence between any two Cauchy location densities $p_{l_1,s}$ and $p_{l_2,s}$ is a scalar function of the
$\chi^2$-divergence~\cite{nielsen2021f}: 
$$
I_f(p_{l_1,s}:p_{l_2,s})=g_f(\chi_s(l_1,l_2)),
$$
we have
$$
I_f(p_{l_1,s}:p_{l_2,s})=g_f\left(\frac{1}{2}\Delta^2_{s^2}\right).
$$
Therefore it follows that
$$
h_f(u)=g_f \left(\frac{1}{2}u^2\right).
$$
See Table~1 of~\cite{nielsen2021f} for several examples of scalar functions corresponding to $f$-divergences.

Now we let $d \ge 2$. 
Contrary to the normal case, it is in general difficult to obtain explicit expressions for $h_f$ in the Cauchy case. 
Here, we give one example to illustrate that difficulty. 
Let $f(u) := (u-1)^2$ (the corresponding divergence is the $\chi^2$ divergence), $d = 3$, and  $F$ as Eq.~(\ref{eq:def-F}). 
Then, 
\[ F(t) = \int_{\mathbb R^3} \frac{p(\|x-t\|^2)^2}{p(\|x\|^2)} dx - 1, \ \ t \in \mathbb{R}^3. \]
By calculations, we find that 
\begin{eqnarray*}
\lefteqn{\int_{\mathbb R} \frac{(1+x_1^2 + x_2^2 + x_3^2)^2}{(1+ (x_1-t)^2 + x_2^2 + x_3^2)^4} \dx_1 =}\\
&& \frac{\pi}{16(1+ x_2^2 + x_3^2)^{3/2}} \left( 8 + \frac{16 t^2}{1+x_2^2+x_3^2} + \frac{5t^4}{(1+x_2^2+x_3^2)^2} \right).
\end{eqnarray*}
Therefore, 
\begin{eqnarray*}
F(t) &=& \frac{1}{\pi} \int_{\mathbb R^2} \frac{\pi}{16(1+ x_2^2 + x_3^2)^{3/2}} \left( 8 + \frac{16 t^2}{1+x_2^2+x_3^2} + \frac{5t^4}{(1+x_2^2+x_3^2)^2} \right) \dx_2 \dx_3 - 1\\
& =& \frac{2}{3} t^2 + \frac{t^4}{8}. 
\end{eqnarray*}
Hence, 
$$
h_f (s) =  \frac{2}{3} s + \frac{s^2}{8}, \ s \ge 0.
$$
Contrary to the normal case (see Table~\ref{tab:MahIsoMVN}), $h_f$ for the Cauchy family is a polynomial. 
This holds also true for the case that $d=5,7, \dots$.

\subsection{Multivariate $f$-divergences as equivalent univariate $f$-divergences}\label{sec:mvuv}

Ali and Silvey~\cite{ali1966general} further showed how to replace a $d$-dimensional $f$-divergence by an equivalent $1$-dimensional $f$-divergence for the fixed covariance matrix normal distributions:

\begin{proposition}[\cite{ali1966general}, Section~6]\label{prop:fdivD1}
Let $\Delta_\Sigma(\mu_1,\mu_2)$ denote the Mahalanobis distance.
Then we have
$$
I_f(p_{\mu_1,\Sigma},p_{\mu_2,\Sigma}) = I_f \left(p_{0,1},p_{\Delta_\Sigma(\mu_1,\mu_2),1}\right).
$$
\end{proposition}

We can show this assertion  by the change of variable $y=\frac{1}{\Delta_\Sigma(\mu_1,\mu_2)} (x-\mu_1)^\top \Sigma^{-1}(\mu_2-\mu_1)$.
Notice that $\Delta_\Sigma(\mu_1,\mu_2)=\Delta_1(0,\Delta_\Sigma(\mu_1,\mu_2))$, and therefore we can write:
\begin{eqnarray*}
I_f(p_{\mu_1,\Sigma}:p_{\mu_2,\Sigma})&=&h_f(\Delta_\Sigma(\mu_1,\mu_2)),\\
&=&h_f(\Delta_1(0,\Delta_\Sigma(\mu_1,\mu_2))),\\
&=& I_f \left(p_{0,1}:p_{\Delta_\Sigma(\mu_1,\mu_2),1}\right).
\end{eqnarray*}

Property~\ref{prop:fdivD1} yields a computationally efficient method to calculate stochastically the $f$-divergences when not known in closed form (eg., the Jensen-Shannon divergence).
We can estimate the $f$-divergence using $s$ samples $x_1,\ldots, x_s$ independently and identically distributed from a propositional distribution $r(x)$ as:
$$
\hat{I}_f[p:q] = \frac{1}{s}\sum_{i=1}^s \frac{1}{r(x_i)} p(x_i) f\left(\frac{q(x_i)}{p(x_i)}\right).
$$
Using we take the propositional distribution $r(x)=p(x)$.
Estimating $f$-divergences between isotropic Gaussians requires $O(sd)$ time. 
Thus Proposition~\ref{prop:fdivD1} allows to shave a factor $d$.

Notice that when $h_f$ is not available in closed-form, We can tabulate the function using Monte Carlo stochastic estimations of $\hat{h}_f(\Delta^2)$.
Moreover, using symbolic regression software, we can fit a formula to the experimental tabulated data (see the plots in Figure~\ref{exp:JSD}):
For example, for the JSD with $\Delta^2\in [\frac{1}{2},5]$, we find that the function
$\tilde{h}_{f_\JS}(u)=\frac{2.06709 u}{u+8.27508}$ approximates well the underlying intractable function
$h_{f_\JS}(u)$ (relative mean error less than $0.1\%$).
This techniques proves useful specially for bounded $f$-divergences like the total variation distance or the Jensen-Shannon divergence.

\begin{figure}
\centering
\includegraphics[width=0.48\textwidth]{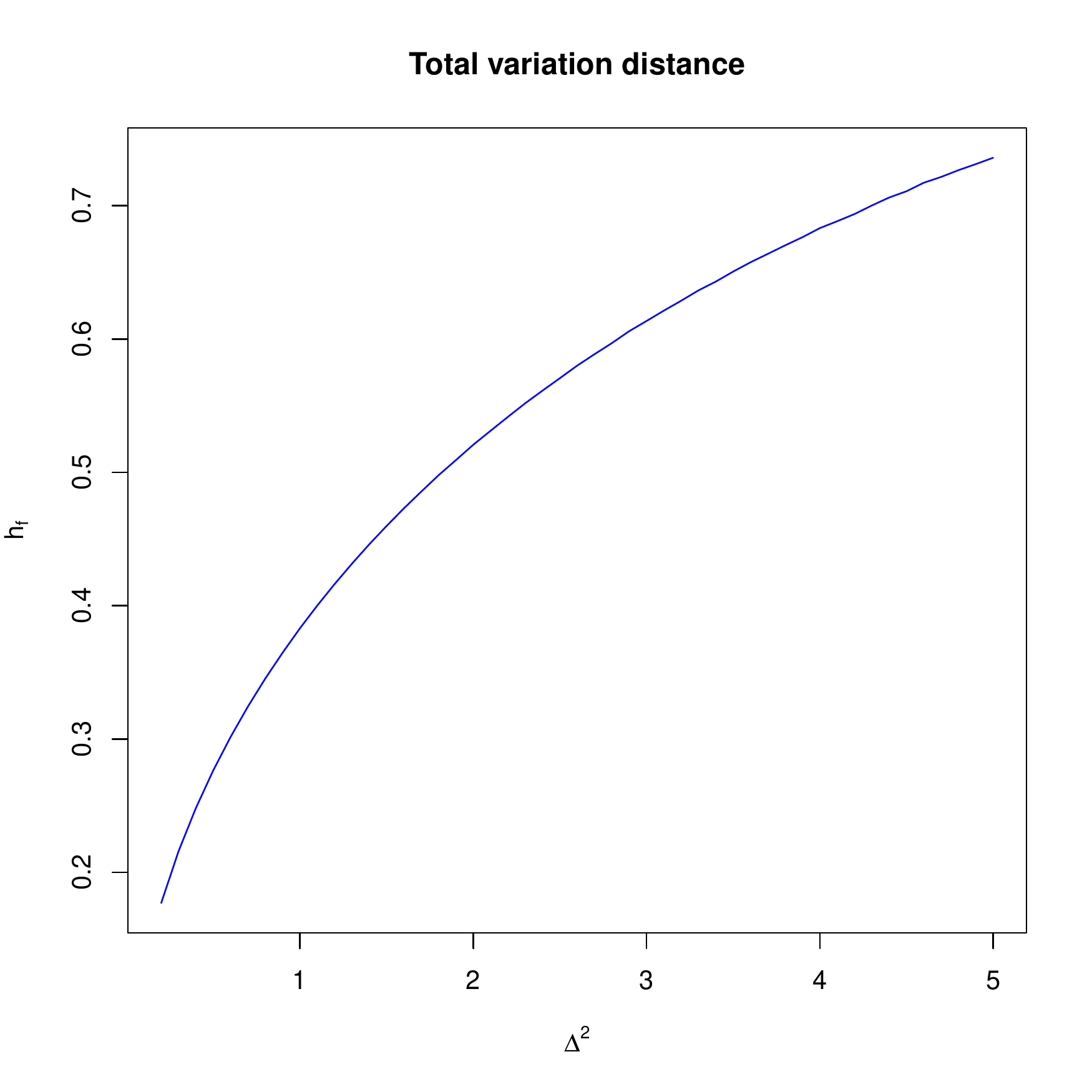}
\includegraphics[width=0.48\textwidth]{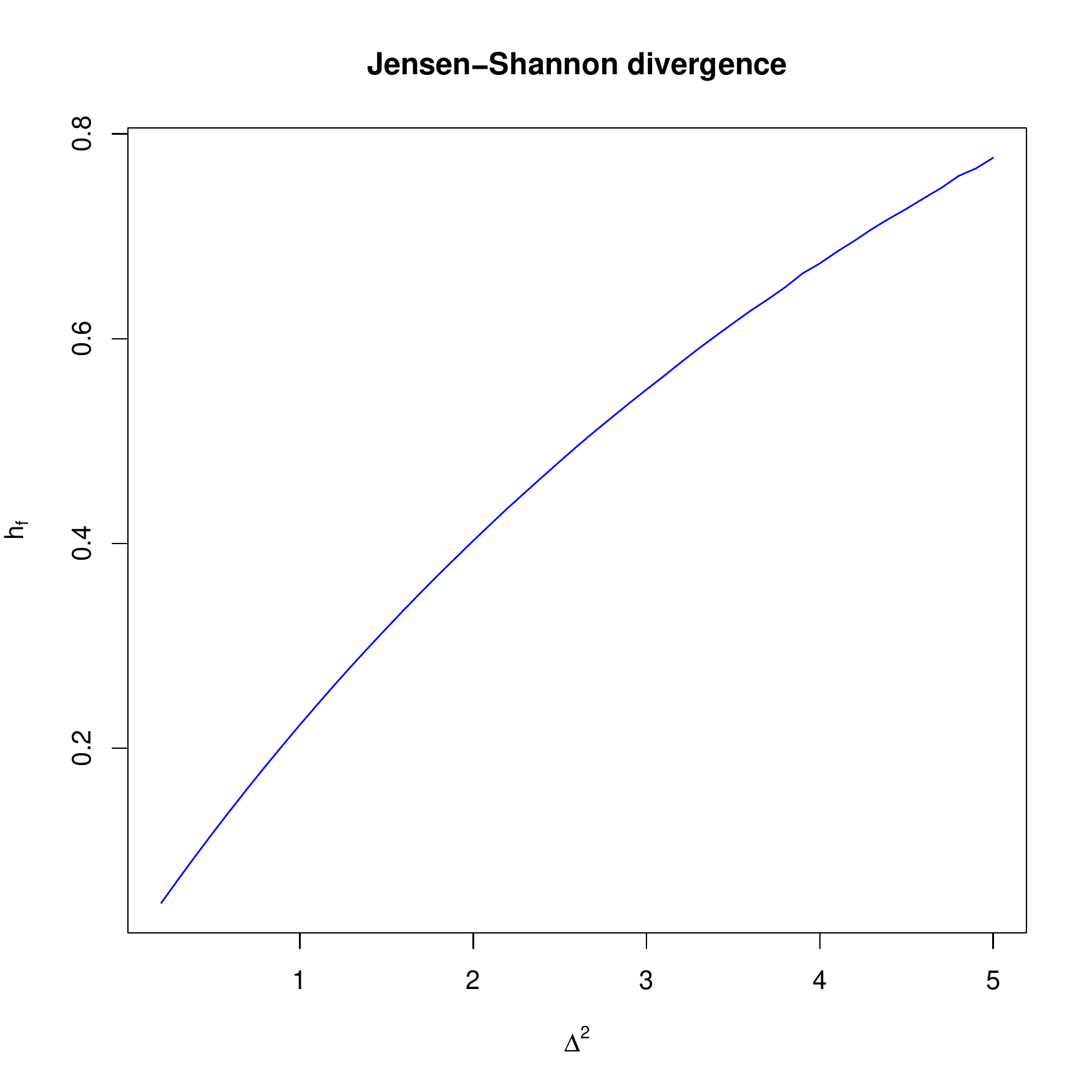}
\caption{Estimating the functions $h_f(\Delta^2)$ by Monte Carlo integration ($y$-axis is $\hat h_f(\Delta^2)$ for $s=10^8$ MC samples) for  the total variation distance (left) and the Jensen-Shannon divergence (right).}\label{exp:JSD}
\end{figure}

\section{$f$-divergences in multivariate scale families}\label{sec:ns}
Consider the scale family $N_\mu=\{p_{\mu,\Sigma}\ :\ \Sigma\succ 0\}$ of $d$-variate MVNs with a prescribed location $\mu\in\bbR^d$.
Ali and Silvey \cite{ali1966general} proved that all $f$-divergences between any two multivariate normal distributions with prescribed mean is an increasing
 function of $\lvert1-\lambda_i\rvert$'s, where the $\lambda_i$'s denote the eigenvalues of $\Sigma_2\Sigma_1^{-1}$.
This property can be proven by using the definition of $f$-divergences of Eq.~\ref{eq:fdiv} and the fact that
$$
\frac{p_{\mu,\Sigma_2\Sigma_1^{-1}}(x)}{p_{\mu,I}(x)}=\prod_{i=1}^d \frac{p_{0,\lambda_i(\Sigma_2\Sigma_1^{-1})}(x_i)}{p_{0,1}(x_i)}.
$$

Thus we have
$$
I_f(p_{\mu,\Sigma_1}:p_{\mu,\Sigma_2})=I_f\left(p_{\mu,I}:p_{\mu,\Sigma_2\Sigma_1^{-1}}\right)=E_f(\lvert 1-\lambda_1\rvert,\ldots,
\lvert 1-\lambda_d\rvert),
$$
where $E_f(\cdot)$ is a $d$-variate totally symmetric function (invariant to permutations of arguments).
Therefore the $f$-divergences  are spectral  matrix  divergences~\cite{kulis2009low}.
In particular, one interesting case is when $E_f(\cdot)$ is a separable function:
$$
E_f(\lvert 1-\lambda_1\rvert,\ldots,\lvert 1-\lambda_d\rvert)=\sum_{i=1}^d e_f(\lvert 1-\lambda_i\rvert).
$$
We shall illustrate these results with the Kullback-Leibler divergence and more generally with the $\alpha$-divergences~\cite{IG-2016} or $\alpha$-Bhattacharyya divergences~\cite{nielsen2011burbea} below:

\begin{itemize}

\item The Kullback-Leibler divergence: The well-known formula of the KLD between two same-mean MVNs is 
 
\begin{eqnarray*}
D_{\KL}(p_{\mu,\Sigma_1}:p_{\mu,\Sigma_2}) &=&\frac{1}{2} \left(\log \det(\Sigma_2\Sigma_1^{-1}) +\tr(\Sigma_2\Sigma_1^{-1}-I)\right).
\end{eqnarray*}
This expression
can be rewritten as $D_{\KL}(p_{\mu,\Sigma_1}:p_{\mu,\Sigma_2})=\sum_{i=1}^d e_{\KL}'(\lambda_i(\Sigma_2\Sigma_1^{-1}))$ where
$$
e_{\KL}'(v)=\frac{1}{2} \left(\log v+ v-1\right),
$$
since $\det(\Sigma_2\Sigma_1^{-1}) =\prod_{i=1}^d \lambda_i(\Sigma_2\Sigma_1^{-1})$ and
$\tr(\Sigma_2\Sigma_1^{-1}-I)=\sum_{i=1}^d  (\lambda_i(\Sigma_2\Sigma_1^{-1})-1)$.
By a change of variable $v=1-u$, we get
\begin{equation}
e_{f_\KL}(u)=\frac{1}{2} \left( \log (1-u)-u\right),
\end{equation}
and $E_{f_\KL}(u_1,\ldots,u_d)=\sum_{i=1}^d e_{f_\KL}(u_i)$ (separable case).
We check that the scalar function $e_{f_\KL}(u)$ is an increasing function of $u$.

\item More generally, let us consider the family of $\alpha$-divergences~\cite{IG-2016}:
$$
D_\alpha(p:q)=\left\{
\begin{array}{ll}
\frac{4}{1-\alpha^2}\left(1-\rho_{\frac{1-\alpha}{2}}(p:q)\right), & \alpha\not\in\{-1,1\}\\
D_\KL(p:q), & \alpha=-1\\
D_\KL(q:p), & \alpha=1.
\end{array}
\right.,
$$
where
$$
\rho_\beta(p:q)=\int p^\beta(x)q^{1-\beta}(x) \dmu(x),
$$
is the skew Bhattacharyya coefficient (a similarity measure also called an affinity measure).
The skew Bhattacharyya distance~\cite{nielsen2011burbea} is $D_\mathrm{Bhat}(p:q)=-\log \rho_\beta(p:q)$.
We have $D_\alpha(q:p)=D_{-\alpha}(p:q)$.
We recover the squared Hellinger divergence when $\alpha=0$ and the 
Neyman $\chi^2$-divergence when $\alpha=3$ (and the Pearson $\chi^2$-divergence   when $\alpha=-3$).
The $\alpha$-divergences are $f$-divergences for the following family $f_\alpha(u)$ of generators:
$$
f_\alpha(u)=\left\{
\begin{array}{ll}
\frac{4}{1-\alpha^2}\left(u-u^{\frac{1+\alpha}{2}}\right), & \alpha\not\in\{-1,1\}\\
-\log u, & \alpha=-1\\
u\log u, & \alpha=1.
\end{array}
\right.
$$

We have the following closed-form formula between two scale normal distributions~\cite{pardo2018statistical} (page 46):
\begin{equation}\label{eq:pardo}
\rho_\beta(p_{\mu,\Sigma_1}:p_{\mu,\Sigma_2})=\frac{\det(\Sigma_1)^{\frac{1-\beta}{2}}\, \det(\Sigma_2)^{\frac{\beta}{2}}}{
\det( (1-\beta)\Sigma_1+\beta\Sigma_2)^{\frac{1}{2}}}
\end{equation}

We can rewrite Eq.~\ref{eq:pardo} as follows:
$$
\det( (1-\beta)\Sigma_1+\beta\Sigma_2 ) =\det( I+\beta(\Sigma_2\Sigma_1^{-1}-I)) \, \det(\Sigma_1),
$$
so that
$$
\rho_\beta(p_{\mu,\Sigma_1}:p_{\mu,\Sigma_2})= \frac{\det(\Sigma_2\Sigma_1^{-1})^{\frac{\beta}{2}}}{\det( I+\beta(\Sigma_2\Sigma_1^{-1}-I) )^{\frac{1}{2}}}.
$$

Using the eigenvalues $\lambda_i$'s for $i\in\{1,\ldots, d\}$ of $\Sigma_2\Sigma_1^{-1}$, we have
$$
\rho_\beta(p_{\mu,\Sigma_1}:p_{\mu,\Sigma_2}) = \prod_{i=1}^d \sqrt{\frac{\lambda_i^\beta}{1+\beta(\lambda_i-1)}}.
$$
Indeed, consider the characteristic polynomial 
$$
p_{\Sigma_2\Sigma_1^{-1}}(x)=\Mdet{xI-\Sigma_2\Sigma_1^{-1}}=\prod_{i=1}^d (x-\lambda_i(\Sigma_2\Sigma_1^{-1})).
$$
We have $\Mdet{\Sigma_2\Sigma_1^{-1}}=p_{\Sigma_2\Sigma_1^{-1}}(0)=(-1)^d\prod_{i=1}^d \lambda_i$ and
$$
\Mdet{I+\beta(\Sigma_2\Sigma_1^{-1}-I)}=\beta^d\Mdet{\left(\frac{1}{\beta}-1\right)\Sigma_2\Sigma_1^{-1}}
=\beta^d p_{\Sigma_2\Sigma_1^{-1}}\left(\frac{1}{\beta}-1\right). 
$$

Thus the $\alpha$-divergences or the Bhattacharyya divergences are increasing functions of the $\lvert 1-\lambda_i\rvert$'s as stated by Ali and Silvey~\cite{ali1966general}.

\end{itemize}

Notice that the bounds on the total variation distance between two multivariate Gaussian distributions with same mean has been investigated in~\cite{TVMVN-2018} but no closed-form formula is known.

Finally, we show that the $f$-divergences between two densities of a scale family are always spectral matrix divergences:

\begin{proposition}\label{prop:spectralfdiv}
For location-scale families $\{p_{\mu, \Sigma}(x) = (\det\Sigma)^{-1/2} p(\Sigma^{-1/2} (x-\mu))\}_{\mu, \Sigma}$ where $p$ is the standard density such that $p(x) = \widetilde p(\lvert x \rvert^2)$ for some $\widetilde p$, 
every $f$-divergence  $I_f(p_{\mu,\Sigma_1}:p_{\mu,\Sigma_2})$ between scale family is a function of the eigenvalues of $\Sigma_1 \Sigma_2^{-1}$.
\end{proposition}

\begin{proof}
We can assume that $\mu=0$. 
By the change-of-variable $x = \Sigma_1^{1/2}y$, 
\[ I_f(p_{\mu,\Sigma_1}:p_{\mu,\Sigma_2}) 
= \int_{\mathbb R^d} f\left(\left(\frac{\det\Sigma_1}{\det\Sigma_2}\right)^{1/2} 
\frac{\widetilde p(x^{\top}\Sigma_2^{-1}x)}{\widetilde p(x^{\top}\Sigma_1^{-1}x)} \right) dx\]
\[ = \int_{\mathbb R^d} f\left(\left(\frac{\det\Sigma_1}{\det\Sigma_2}\right)^{1/2} 
\frac{\widetilde p\left(y^{\top}\Sigma_1^{1/2}\Sigma_2^{-1}\Sigma_1^{1/2}y\right)}{\widetilde p(\lvert y \rvert^2)} \right) dy. \]

Since $\Sigma_1$ and $\Sigma_2$ are both symmetric matrices, $\Sigma_1^{1/2}$ and $\Sigma_2^{-1}$ are both symmetric, and hence,  $\Sigma_1^{1/2}\Sigma_2^{-1}\Sigma_1^{1/2}$ is also symmetric. 
Hence it is diagonalizable by an orthogonal matrix and there exist real eigenvalues $\lambda_1, \dots, \lambda_d$ and 
$\det(\Sigma_1^{1/2}\Sigma_2^{-1}\Sigma_1^{1/2}) = \lambda_1 \cdots \lambda_d$. 
Since $\Sigma_2$ is positive-definite, 
$\Sigma_2^{-1}$ is also positive-definite. 
By this and the fact that $\Sigma_1^{1/2}$ is symmetric, 
$\Sigma_1^{1/2}\Sigma_2^{-1}\Sigma_1^{1/2}$ is positive-definite, and hence, 
$\lambda_1, \dots, \lambda_d$ are all positive. 
Now we recall that the set of eigenvalues of $\Sigma_1^{1/2}\Sigma_2^{-1}\Sigma_1^{1/2}$ is identical with the set of eigenvalues of $\Sigma_1^{1/2}\Sigma_1^{1/2}\Sigma_2^{-1} = \Sigma_1 \Sigma_2^{-1}$, which is a well-known result in linear algebra. 
Hence,
\[ \int_{\mathbb R^d} f\left(\left(\frac{\det\Sigma_1}{\det\Sigma_2}\right)^{1/2} 
\frac{\widetilde p\left(y^{\top}\Sigma_1^{1/2}\Sigma_2^{-1}\Sigma_1^{1/2}y\right)}{\widetilde p(\lvert y \rvert^2)} \right) dy \]
\[ = \int_{\mathbb R^d} f\left(\left(\lambda_1 \cdots \lambda_d\right)^{1/2} 
\frac{\widetilde p\left(\lambda_1 y_1^2 + \cdots + \lambda_d y_d^2\right)}{\widetilde p(\lvert y \rvert^2)} \right) dy_1 \cdots dy_d.\]
\end{proof}

\bibliographystyle{plain}
\bibliography{V3BIB}

\end{document}